\documentclass[12pt,a4paper]{article}
\usepackage[margin=1in,nohead]{geometry}
\usepackage[T1]{fontenc}
\usepackage[utf8]{inputenc}
\usepackage{ae}
\usepackage{graphicx}
\usepackage[american]{babel}
\usepackage{amsfonts}
\usepackage{amsthm}
\usepackage{amsmath}
\usepackage{amssymb}
\usepackage{enumerate}
\usepackage{wasysym}
\usepackage{paralist}
\usepackage{multirow}

\newtheorem{Theorem}{Theorem}
\newtheorem{Lemma}[Theorem]{Lemma}

\newcommand{\dist}{\operatorname{dist}}
\newcommand{\conv}{\operatorname{conv}}
\newcommand{\Ima}{\operatorname{Im}}

\newcommand{\N}{\mathbb{N}}
\newcommand{\Z}{\mathbb{Z}}

\newcommand{\R}{\mathbb{R}}

\renewcommand{\:}{\colon}

\title{Block partitions in higher dimensions}
\author{Endre Cs\'oka
\thanks{Alfr\'ed R\'enyi Institute of Mathematics. Supported by the NRDI grant KKP~138270.}
}
\date{}

\begin{document}

\maketitle

\begin{abstract}
Consider a set $X\subseteq \mathbb{R}^d$ which is 1-dense, namely, it intersects every open unit ball. 
We show that we can get from any point to any other point in $\mathbb{R}^d$ in $n$ steps so that the intermediate points are in $X$, and the discrepancy of the step vectors is at most $2\sqrt{2}$. 
Or formally, $$\sup\limits_{\substack{n\in \mathbb{Z}^+,\ t\in \mathbb{R}^d\\ X\text{ is 1-dense}}}\,\, \inf\limits_{\substack{p_1,\ldots, p_{n-1}\in X\\ p_0=\underline{0},\ p_n=t}}\,\, \max\limits_{0\leq i<j<n} \big\|(p_{i+1}-p_i)-(p_{j+1}-p_j)\big\|\leq 2\sqrt{2}.$$

\end{abstract}

\section{Introduction}
A block of a sequence of real numbers $(a_1, \ldots, a_m)$ is a subsequence $(a_i, \ldots, a_k)$ of consecutive elements, and a block partition of the sequence is a list of blocks that contain each element exactly once. 
The size of a block is the sum of its elements. 
In \cite{BG15} it was shown that given an $n\in \N$ and a sequence of numbers in $[0,1]$, the sequence admits a block partition into $n$ blocks such that the difference between any two block sizes is at most 1. 
This result was generalized in \cite{BCKT18} to sequences of vectors in $\R^d$ by showing lower and upper bounds on the \emph{discrepancy} of such sequences. 
The discrepancy of $(p_0, \ldots, p_n)\in \R^{d\times (n+1)}$ is defined as $\delta(p_0, \ldots, p_n)=\max\limits_{0\leq i<j<n} \big\|(p_{i+1}-p_i)-(p_{j+1}-p_j)\big\|$. 
A set $X\subseteq \R^d$ is \emph{1-dense} if $\forall y\in \R^d\: \dist(y,X)< 1$. 
Then 
$$D_d=\sup\limits_{\substack{n\in \Z^+, t\in \R^d\\ X_1,\ldots, X_{n-1}\,\, \text{are 1-dense}}}\,\, \inf\limits_{\substack{p_1,\ldots, p_{n-1}\in X\\ p_0=\underline{0}, p_n=t}}\,\, \max\limits_{0\leq i<j<n} \big\|(p_{i+1}-p_i)-(p_{j+1}-p_j)\big\|=$$
$$\sup\limits_{\substack{n\in \Z^+, t\in \R^d\\ X_1,\ldots, X_{n-1}\,\, \text{are 1-dense}}}\,\, \inf\limits_{\substack{p_1,\ldots, p_{n-1}\in X\\ p_0=\underline{0}, p_n=t}}\,\, \delta(p_0, \ldots, p_n).$$
The above mentioned result from \cite{BG15} translates to $D_1=2$. 
Moreover, the inequalities $D_1\leq D_2\leq D_3\leq \cdots \leq 4$ are obvious. 
In \cite{BCKT18} the nontrivial lower estimate $1+\sqrt{2}\leq D_2$ was shown. 
For some related problems, see \cite{LPS93}. 

We define a variant of the discrepancy $\delta'(p_0, \ldots, p_n)=2\cdot\min\limits_{s\in \R^d}\max\limits_{0\leq i<n} \big\|(p_{i+1}-p_i)-s\big\|$. 
By replacing $\delta$ with $\delta'$ in the definition of $D_d$, we obtain $D_d'$. Clearly $\delta\leq \delta'$ for all tuples, and in particular, $D_d\leq D_d'$. 
The following theorem is the main result of the current paper. 

\begin{Theorem}\label{thm:firstform}
For all $d,n\in\N$, $t\in \R^d$ and $X_1, \ldots, X_{n-1} \subseteq \R^d$ 1-dense sets there exist a tuple $(p_0, \ldots, p_n)$ and $s\in \R^d$ such that $p_0=\underline{0}, p_n=t$, $p_i\in X_i$ for all $1\leq i \leq n-1$, and for all $0\leq i \leq n-1$ we have $\big\|(p_{i+1}-p_i)-s\big\|\leq \sqrt{2}$. 
In particular, $D_d\leq D'_d\leq 2\sqrt{2}$.
\end{Theorem}

It was already mentioned in \cite{BCKT18} that it has no effect on $D_d$ (and similarly on $D_d'$) if we apply any of the restrictions that either $t=\underline{0}$ or that all the $X_i$ coincide with a 1-dense set $X$. 
We phrase all \textbf{equivalent forms of Theorem~\ref{thm:firstform}} obtained in these ways. 

\begin{Theorem}\label{thm:Xmain}
For all $d,n\in\N$, $t\in \R^d$ and $X \subseteq \R^d$ 1-dense set there exist a tuple $(p_0, \ldots, p_n)$ and $s\in \R^d$ such that $p_0=\underline{0}, p_n=t$, $p_i\in X$ for all $1\leq i \leq n-1$, and for all $0\leq i \leq n-1$ we have $\big\|(p_{i+1}-p_i)-s\big\|\leq \sqrt{2}$. 
\end{Theorem}

\begin{Theorem}\label{thm:genmain}
For all $d,n\in\N$ and $X_1, \ldots, X_{n-1} \subseteq \R^d$ 1-dense sets there exist a tuple $(p_0, \ldots, p_n)$ and $s\in \R^d$ such that $p_0=p_n=\underline{0}$, $p_i\in X_i$ for all $1\leq i \leq n-1$, and for all $0\leq i \leq n-1$ we have $\big\|(p_{i+1}-p_i)-s\big\|\leq \sqrt{2}$. 
\end{Theorem}

Theorems~\ref{thm:Xmain} and \ref{thm:genmain} are special cases of Theorem~\ref{thm:firstform}. For the other directions, notice that the problem in Theorem 1 remains equivalent if we replace $t$ with $t'$ and translate each $X_i$ by $-\frac{i}{n} \cdot (t' - t)$. For Theorem~\ref{thm:Xmain}, we should use a $t'$ large enough, and for Theorem~\ref{thm:genmain}, we should use $t' = 0$.

We are going to prove Theorem~\ref{thm:genmain}. 

\section{Proof of the main result}

We start with some simple lemmas concerning convex geometry and topology. 
The closed ball with center $c$ and radius $r$ in $\R^d$ is denoted by $B(c,r)$. 
More generally, the $r$-neighborhood of a set $C$, i.e., the set of points at distance at most $r$ from a set $C$ is denoted by $B(C,r)$. 

\begin{Lemma}\label{lem:convhull}
If $y_1,\ldots, y_k\in B(c,1)$ for some $c\in \R^d$, then 
$$B(\conv(y_1,\ldots, y_k),1)\subseteq \bigcup\limits_{i=1}^k B\big(y_i, \sqrt{2}\big)$$
\end{Lemma}
\begin{proof}
Let $y\in B(\conv(y_1,\ldots, y_k),1)$. 
Let $z\in \conv(y_1,\ldots, y_k)$ be the closest point to $y$. 
Then $\dist(y,z)\leq 1$, and the segment $\overline{yz}$ is perpendicular to the face $F$ of $\conv(y_1,\ldots, y_k)$ containing $z$. 
As $z\in B(c,1)$, we have that $z\notin \conv(B(c,1)\setminus B(z,1))$, and in particular $z\notin \conv(F\setminus B(z,1))$. 
Thus there is an index $1\leq i\leq k$ such that $y_i\in F\cap B(z,1)$, and then $\dist(y,y_i)= \sqrt{\dist(y,z)^2+\dist(z,y_i)^2}\leq \sqrt{2}$. 
\end{proof}

We note that the radius $\sqrt{2}$ is optimal for $d\geq 2$: if $y_1$ and $y_2$ are antipodal points on the boundary of the closed unit ball $B(c,1)$, then there is a point $y$ on the same unit sphere such that $\dist(y,y_1)=\dist(y,y_2)=\sqrt{2}$. 
Clearly $y\in B(\conv(y_1, y_2),1)$ as the center $c$ is in $\conv(y_1, y_2)$. 

\begin{center}
\includegraphics[scale=0.8]{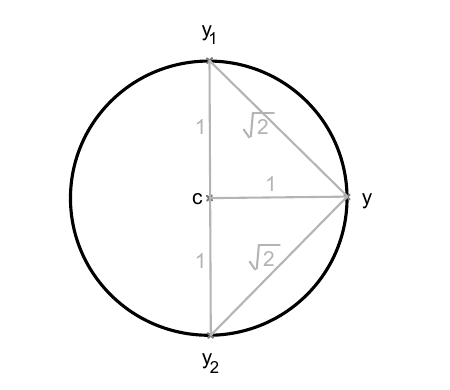}
\end{center}

For $d=1$, the value $\sqrt{2}$ can be replaced with 1. 
This is the crucial difference between one and higher dimensions, which yields the different estimates $D_1'\leq 2$ in the one-dimensional case and $D_d'\leq 2\sqrt{2}$ for $d\geq 2$. 

\begin{Lemma}\label{lem:binary}
If $X\subseteq \R^d$ is 1-dense, then there is a continuous function $g\: \R^d \rightarrow \R^d$ such that for all $c\in \R^d$ we have $g(c)\in \conv(X\cap B(c,1))$. 
\end{Lemma}
\begin{proof}
The set $X\cap B(c,1)$ is nonempty because $X$ is 1-dense. 
Define $g(c)$ as the weighted centroid of the set $X\cap B(c,1)$ where a point $y\in X\cap B(c,1)$ has weight $1-\big\|c-y\big\|$. 
As the expression $1-\big\|c-y\big\|$ is zero if $y$ is on the boundary of the closed ball $B(c,1)$, and the total weight is never zero, the weighted centroid is a continuous function. 
\end{proof}

\begin{Lemma}\label{lem:Brouwer}
Let $R>0$. 
If $f\: B(\underline{0},R)\rightarrow \R^d$ is continuous and $f(x)=x$ for all $\big\|x\big\|=R$, then $\underline{0}\in \Ima(f)$.
\end{Lemma}
\begin{proof}
Seeking for a contradiction, assume that $\underline{0}\notin \Ima(f)$. 
As $\Ima(f)$ is compact, the continuous function $\big\|f(x)\big\|$ has a positive global minimum and maximum. 
Hence, the function $g\: B(\underline{0},R)\rightarrow B(\underline{0},R)$ defined by $g(x)=(R/\big\|f(x)\big\|)\cdot f(x)$ is continuous. 
Then the function $h\: B(\underline{0},R)\rightarrow B(\underline{0},R)$ defined by $h(x)=g(-x)$ is also continuous. 
According to Brouwer's fixed point theorem \cite{Brouwer1912}, the mapping $h$ has a fixed point $x_0\in B(\underline{0},R)$, that is, $h(x_0)=x_0$. 
As every vector in the image of $h$ has length $R$, such a fixed point $x_0$ must have length $R$. 
By assumption, $f(-x_0)=-x_0$, and consequently, $h(x_0)=-x_0$, a contradiction.
\end{proof}

%
\noindent\textbf{Proof of Theorem~\ref{thm:genmain}.}
We define a sequence of functions $h_n\: \R^d\rightarrow \R^d$ recursively for all $n\in N\cup\{0\}$. 
Let $h_0(s)=\underline{0}$, and for all $n\in \N$ let $h_n(s)=g_n(h_{n-1}(s)+s)$, where $g_n$ is the continuous function in Lemma~\ref{lem:binary} applied to $X=X_n$. 
Thus $h_n$ is continuous, and Lemma~\ref{lem:binary} also implies $h_n(s)\in B(ns,n)$. 
The purpose of the next paragraph is to formally prove that the $h_n$ image of the boundary of the closed ball $B(\underline{0},2)$ surrounds the origin, and conclude that the origin must be in the image of $h_n$. 

We define a continuous function $f\: B(\underline{0},3)\rightarrow \R^d$ in the following way. 
Let $f$ agree with $h_n$ on $B(\underline{0},2)$, and if $\big\|s\big\|=3$ then let $f(s)=s$.
We extend this partial function linearly, namely, if $2\leq \big\|s\big\|\leq 3$, then let $f(s)=(3-\big\|s\big\|)\cdot h_n(2s/\big\|s\big\|)+ (\big\|s\big\|-2)\cdot s$. 
Importantly, $f(s)\neq \underline{0}$ if $2\leq \big\|s\big\|\leq 3$, as the scalar product $\big\langle f(s), s\big\rangle$ is positive. 
Indeed, $\langle s, s\rangle>0$, and $\big\langle h_n(2s/\big\|s\big\|),  s\big\rangle$ is also positive: 
$$\big\langle h_n(2s/\big\|s\big\|),  s\big\rangle \geq \big\langle2n s/\big\|s\big\|, s\big\rangle - n\big\|s\big\| = n\big\|s\big\|>0.$$
According to Lemma~\ref{lem:Brouwer}, there is an $s_0\in B(\underline{0},3)$ with $f(s_0)=\underline{0}$. 
As this root must be in the domain $B(\underline{0},2)$, we have $h_n(s_0)=\underline{0}$. 

Given an $s\in \R^d$ we say that a vector $t\in \R^d$ is $s$-reached in $n$ steps (from the origin) if there is a tuple $(p_0, \ldots, p_n)$ such that $p_0=\underline{0}, p_n=t$, $p_i\in X_i$ for all $1\leq i \leq n-1$, and for all $0\leq i \leq n-1$ we have $\big\|(p_{i+1}-p_i)-s\big\|\leq \sqrt{2}$. 
If $t$ is $s$-reached in $n$ steps for some $s\in \R^d$ then we say that $t$ is reached in $n$ steps. 
Our goal is to show that the origin $t=\underline{0}$ is reached in $n$ steps for all $n\in \N$. 
According to the conclusion of the previous paragraph, it is enough to show that $h_n(s)$ is $s$-reached in $n$ steps: this would conclude the proof of Theorem~\ref{thm:genmain}. 
To this end, we show the following more general claim by induction on $n$. 
The fact that it is indeed more general follows from Lemma~\ref{lem:binary}: note that $h_n(s)=g_n(h_{n-1}(s)+s)\in \conv(X_n\cap B(h_{n-1}(s)+s,1))$.

\textit{Claim.} Every point in $\conv(X_n\cap B(h_{n-1}(s)+s,1))$ is $s$-reached in $n$ steps. 

\textit{Proof of Claim.} Let $n=1$ for the base case of the induction. 
That is, we need to prove that every point $y\in \conv(X_1\cap B(s,1))$ is $s$-reached in $1$ step. 
As $y\in B(s,1)$, the condition obviously applies to the tuple $(p_0,p_1)=(\underline{0},y)$.

Assume that the claim holds for $n\in \N$; we show it for $n+1$. 
That is, we need to prove that any $y\in \conv(X_{n+1}\cap B(h_{n}(s)+s,1))$ is $s$-reached in $n+1$ steps. 
According to Lemma~\ref{lem:binary}, we have $h_n(s)\in \conv(X_n\cap B(h_{n-1}(s)+s,1))$. 
Then $h_n(s)$ can be written as the convex combination of finitely many points $x_1,\ldots, x_k\in X_n\cap B(h_{n-1}(s)+s,1)$. 
In particular, the points $y_1=x_1+s,\ldots, y_k=x_k+s$ are contained in the closed unit ball $B(h_{n-1}(s)+2s,1)$. 
As $y$ is in the 1-neighborhood of the point $h_{n}(s)+s$, it is also in the 1-neighborhood of the set $\conv(y_1,\ldots, y_k)$. 
Hence, the conditions of Lemma~\ref{lem:convhull} apply to $y_1,\ldots, y_k$ and $y$. 
Thus there is an index $i$ such that $y\in B\big(y_i, \sqrt{2}\big)$. 

By the induction hypothesis, $x_i$ is $s$-reached in $n$ steps. 
Let $\underline{0}=p_0, p_1, \ldots, p_{n-1}, p_n=x_i$ be a witness to this fact. 
Then putting $p_{n+1}=y$ yields a tuple $(p_0, p_1, \ldots, p_n, p_{n+1})$ that $s$-reaches $y$ in $n+1$ steps: the only condition to check is $\big\|(p_{n+1}-p_n)-s\big\|\leq \sqrt{2}$, which is obvious since $(p_{n+1}-p_n)-s=(y-x_i)-s=y-y_i$ and $y\in B\big(y_i, \sqrt{2}\big)$. 
\qed
\bigskip

Using the same proof technique, we can obtain the optimal upper estimate $D_1'\leq 2$ in the 1-dimensional case. 
As we mentioned earlier, the constant $\sqrt{2}$ in the one-dimensional version of Lemma~\ref{lem:convhull} can be replaced with 1. 
Moreover, Lemma~\ref{lem:Brouwer} simplifies to the more elementary Intermediate Value Theorem (or Bolzano's theorem). 
Thus the result of \cite{BG15} follows from the above argument. 

\bibliographystyle{plain}
\bibliography{refs}

\begin{thebibliography}{1}

\bibitem{BCKT18}
Imre B\'ar\'any, Endre Cs\'oka, Gyula K\'arolyi, and G\'eza T\'oth.
\newblock Block partitions: an extended view.
\newblock {\em Acta Mathematica Hungarica}, 155:36--46, 2018.

\bibitem{BG15}
Imre B\'ar\'any and Victor~S. Grinberg.
\newblock Block partitions of sequences.
\newblock {\em Israel Journal of Mathematics}, 206:155--164, 2015.

\bibitem{Brouwer1912}
Luitzen Egbertus~Jan Brouwer.
\newblock {\"{U}}ber {A}bbildung von {M}annigfaltigkeiten.
\newblock {\em Mathematische Annalen}, 71:97--115, 1912.

\bibitem{LPS93}
Mario Lucertini, Yehoshua Perl, and Bruno Simeone.
\newblock Most uniform path partitioning and its use in image processing.
\newblock {\em Discrete Applied Mathematics}, 42(2):227--256, 1993.

\end{thebibliography}

\end{document}